\documentclass{amsart}
\usepackage{amsmath,amssymb,enumerate,graphicx,url,amscd, color}

\usepackage[utf8]{inputenc}

\def\del{\partial}
\def\BVi{\operatorname{BV}_\infty}
\def\Dh{\widehat{D}}
\def\dhat{\widehat{d}}
\def\g{\mathfrak g}
\def\h{\mathfrak h}
\def\IBLi{\operatorname{IBL}_\infty}
\def\Li{L_\infty}
\def\m{\mathfrak m}

\def\sfe{{\sf e}}
\def\F{\mathcal{F}}
\def\M{\mathcal{M}}
\def\OO{\mathcal{O}}
\def\P{\mathcal{P}}
\def\nz{\mathbb Z}
\DeclareMathOperator{\Ad}{Ad}
\DeclareMathOperator{\ad}{ad}
\DeclareMathOperator{\BViAlg}{\mathcal{BV}_\infty-\mathcal{A}\textit{lg}}
\DeclareMathOperator{\CLAlg}{\mathcal{CLA}\textit{lg}}
\DeclareMathOperator{\CDGCAlg}{\mathcal{CDGCA}\textit{lg}}
\DeclareMathOperator{\CoCom}{coCom}
\DeclareMathOperator{\Coder}{Coder}
\DeclareMathOperator{\CoFrob}{coFrob}
\DeclareMathOperator{\Com}{Com}
\DeclareMathOperator{\End}{\mathcal{E}\!\textit{nd}}
\DeclareMathOperator{\FLAff}{\mathcal{FLA}\textit{ff}}
\DeclareMathOperator{\Frob}{Frob}
\DeclareMathOperator{\Gr}{Gr}
\DeclareMathOperator{\Hom}{Hom}
\DeclareMathOperator{\Ima}{Im}
\DeclareMathOperator{\op}{op}
\DeclareMathOperator{\MC}{MC}
\DeclareMathOperator{\LiAlg}{\Li-\mathcal{A}\textit{lg}}
\DeclareMathOperator{\Mor}{Mor}
\DeclareMathOperator{\BViSp}{\mathcal{BV}_\infty-\mathcal{S}\textit{p}}
\DeclareMathOperator{\PFDGMan}{\mathcal{PFDGM}\textit{an}}
\DeclareMathOperator{\QM}{QM}

\DeclareMathOperator{\Set}{\mathcal{S}\!\textit{et}}
\DeclareMathOperator{\Spec}{Spec}
\providecommand{\abs}[1]{\lvert#1\rvert}

\newtheorem{theorem}{Theorem}[section]
\newtheorem{metatheorem}[theorem]{Metatheorem}

\newtheorem{corollary}[theorem]{Corollary}

\newtheorem{innercustomthm}{Metatheorem}
\newenvironment{metatheorem-quoted}[1]
  {\renewcommand\theinnercustomthm{#1}\innercustomthm}
  {\endinnercustomthm}

\theoremstyle{definition}
\newtheorem{definition}[theorem]{Definition}
\newtheorem{example}[theorem]{Example}

\theoremstyle{remark}
\newtheorem*{remark}{Remark}

\title{Quantizing Deformation Theory II}

\author[A. A. Voronov]{Alexander A. Voronov} \email{voronov@umn.edu}
\address {School of Mathematics\\University of Minnesota\\
  Minneapolis, MN 55455, USA, and Kavli IPMU (WPI), UTIAS, University
  of Tokyo, Kashiwa, Chiba 277-8583, Japan}

\date{July 6, 2018}

\dedicatory{Dedicated to my teacher Yuri Ivanovich Manin on the
  occasion of his eightieth birthday.}

\begin{document}

\begin{abstract}
A quantization of classical deformation theory, based on the
Mau\-rer-Cartan Equation $dS + \frac{1}{2}[S,S] = 0$ in dg-Lie
algebras, a theory based on the Quantum Master Equation $dS + \hbar
\Delta S + \frac{1}{2} \{S,S\} = 0$ in dg-BV-algebras, is
proposed. Representability theorems for solutions of the Quantum
Master Equation are proven. Examples of ``quantum'' deformations are
presented.
\end{abstract}

\maketitle

\tableofcontents

\section*{Introduction}

Yuri I. Manin has always been fascinated with the concept of
quantization. Observing the chromatic spectrum of his work over the
years, I have become more and more convinced that you may quantize
more than you expect.

In this paper, I suggest an approach to quantizing deformation
theory. Neither the idea, nor the terminology is new: I am referring
to John Terilla's paper \cite{terilla} on \emph{Quantizing Deformation
  Theory}, precluded by his work \cite{park-terilla-tradler} with
Jae-Suk Park and Thomas Tradler. This partially explains the title of
the current paper, which I view as a complement to Terilla's work. I
hint on a relation in the last section of this paper. I believe
strongly that these works are two tips of one and the same iceberg.

I have been running around disseminating vague ideas of quantum
deformation theory for a few years since John hooked me on quantizing
deformation theory during the historic Northeast blackout of
2003. Alas, it is a pity it took me so long to get these ideas
crystallized, but now I am content with their shining, albeit somewhat
superficial.

% By quantizing deformation theory, we get deformations with two
% parameters, the deformation parameters as wells as the quantization
% parameter. The two parameters do not stand on an equal footing
% mathematically and semantically. Two-parameter deformations have
% been advocated by Edward Witten.

\subsection*{Conventions on graded algebra and geometry}

We will work over a ground field $k$ of characteristic zero. In
particular, the symbol $\otimes$ will mean $\otimes_k$ by default. The
translation $V[n]$ of a graded vector space $V = \bigoplus_{n \in \nz}
V^n$ is the same vector space with a redefined degree: $V[n]^p :=
V^{p+n}$. The dual $V^*$ of a graded vector space $V$ is understood as
the direct sum of the duals of its graded components, graded in such a
way that the natural pairing $V^* \otimes V \to k$ is
grading-preserving.  In particular, $(V[n])^* = V^*[-n]$. One can also
write $V^* = \hom_k (V, k)$ with $\hom_k (V,V')$ being the internal
Hom in the category of graded vector spaces, as opposed to $\Hom_k (V,
V') = \hom_k (V,V')^0$, the vector space of degree-preserving linear
maps. By default, the degree of a homogeneous tensor is the sum of the
degrees of its factors.  Differentials $d$ are assumed to have degree
1 by default: $\abs{d} = 1$. However, the BV operator $\Delta$ will
have degree $-1$. All associative algebras are assumed to be unital
and all coassociative coalgebras to be counital.

For the purpose of this paper, we will mostly work with \emph{pointed
  formal graded manifolds} that are actually pointed formal
\emph{graded affine spaces}. These are determined by graded symmetric
coalgebras of the type $S(V)$, where $V$ is a graded vector space. The
(graded) cocommutative, coassociative comultiplication on $S(V)$ is
taken to be the standard shuffle comultiplication. We think of the
linear dual algebra $S(V)^*$ as the algebra of functions in a formal
neighborhood of 0 in $V$. The \emph{basepoint}, which corresponds to
the origin 0 in the vector space $V$, is given by the coaugmentation
$k = S^0(V) \to S(V)$.

A \emph{morphism $V \to W$ of pointed formal graded manifolds} in our
restricted, linear category is just a morphism $S(V) \to S(W)$ of
coalgebras respecting the coaugmentations. Since the coalgebra $S(W)$
is cofree (in the category of conilpotent cocommutative coalgebras),
such a morphism is determined by a degree-zero linear map $S(V) \to
W$. Compatibility with coaugmentations forces this linear map to
vanish on $k = S^0(V)$.

When we talk about a linear \emph{pointed formal differential graded
  $($dg-$\,)$ manifold} $V$, we assume that it is a linear pointed
formal graded manifold $V$ endowed with a \emph{differential},
\emph{i.e}., the structure coalgebra $S(V)$ is endowed with a
\emph{codifferential}, a degree-one, $k$-linear coderivation $D$ of
satisfying $D^2 = 0$ and vanishing on $S^0(V)$. A codifferential, like
any coderivation on a cofree cocommutative coalgebra, is determined by
a degree-one linear map $S^{>0}(V) \to V$, the projection of the
coderivation $D$ to the space $V$ of cogenerators of the coalgebra
$S(V)$.

\emph{Morphisms of pointed formal dg-manifolds} have to respect
differentials, \emph{i.e}., the corresponding coalgebra morphisms have
to respect the structure codifferentials.

Typically, a deformation functor is defined on the category of local
Artin rings. We find it more convenient to work with slightly more
general complete local rings (or algebras). Let $(R,\m)$ be a
\emph{complete local $k$-algebra}, that is to say, a local $k$-algebra
$R$ with a maximal ideal $\m$ such that the canonical ring
homomorphisms $k \to R/\m$ to the residue field and $R \to
\varprojlim_n R/\m^n$ to the completion of $R$ in the $\m$-adic
topology are isomorphisms. Given a dg-vector space $V$ and a complete
local algebra $(R,\m)$, we will be considering \emph{completed tensor
  products}, such as
\[
V \; \widehat{\otimes} \; \m := \varprojlim_n V \otimes \m/\m^n.
\]
We say that a complete local $k$-algebra $(R,\m)$ is \emph{of finite
  type} if all the quotients $R/\m^n$, $n \ge 1$, are
finite-dimensional over $k$. We will also associate a pointed formal
dg-manifold $\Spec R$ to a complete local $k$-algebra $(R,\m)$ of
finite type. This formal pointed manifold is determined by the natural
conilpotent coalgebra structure on the continuous linear dual $R^*$ of
$R$, along with the coaugmentation $k \to R^*$, dual to the
augmentation $R \to R/\m = k$. We will set the dg structure on $\Spec
R$ to be given by the zero codifferential on $R^*$. Note that this
\emph{pointed formal dg-manifold} is not linear, but rather
\emph{affine} in the scheme-theoretic sense.

More general pointed formal dg-manifolds are treated in
\cite{bashkirov-voronov3}.

\subsection*{Disclaimer}

Sometimes people refer to quantum deformations in the context of
deformations associated historically with quantum field theory,
especially when the deformation parameter is hidden within the
variable $q$. From the point of view of this paper, many such
deformations would still be classical. But you never know. For
example, a theorem of C.~Teleman \cite{teleman:semi-simple} states
that higher-genus Gromov-Witten invariants can be reconstructed from
the quantum cup product in the semisimple case. The quantum cup
product is a classical deformation of the usual cup product. We argue
that deformation theory of algebraic structures associated to higher
genera is intrinsically quantum, cf.\ Section
\ref{overture}. Likewise, quantum groups would be classical
deformations from our point of view. However, quantum groups are
closely related to Lie bialgebras, whose deformation theory should be
quantum, cf.\ Section~\ref{IBL-qd}.  Deformation quantization
\cite{kontsevich:DQ}, given its relation to moduli spaces of algebraic
curves, could have an incarnation within Quantum Deformation Theory,
but it is still a classical deformation and one does not need to evoke
moduli spaces of higher genera to do it.

\subsection*{Acknowledgments}

My deepest gratitude goes to my teacher Yuri Ivanovich Manin, whose
advice extended well beyond my Ph.D. thesis and affected all my
thoughts and writings ever since. I am also grateful to my faithful
coauthors Denis Bashkirov, Murray Gerstenhaber, Andrey Lazarev, Martin
Markl, and Jim Stasheff for inspiration and teaching me a few useful
things. I thank Ricardo Campos and Jim Stasheff in particular for
comments on earlier versions of the manuscript. My work was supported
by World Premier International Research Center Initiative (WPI), MEXT,
Japan, and a Collaboration grant from the Simons Foundation
(\#282349).

\section{Classical theory: MCE, CME \& deformation theory}

\subsection{The main player of deformation theory}

Let me start with the following, hopefully contentious, statement.

\begin{metatheorem}
\label{meta-cl}
Every reasonable deformation problem in mathematics comes from a
dg-Lie algebra.
\end{metatheorem}

The proof of this statement can easily be demonstrated by
contradiction: If there is a deformation problem that does not come
from a dg Lie algebra, the problem is obviously unreasonable. On a
more serious note, the metatheorem presents a long and important
development stemming from the work of Gerstenhaber
\cite{gerstenhaber}, Schlessinger-Stasheff \cite{s-s}, Goldman-Millson
\cite{g-m}, Deligne \cite{deligne}, Kontsevich \cite{kontsevich:DQ}
and a few others, who showed that, on the one hand, many major
deformation theories, such as those of complex manifolds or
associative algebras, come from a corresponding dg Lie algebra and, on
the other hand, can be completely described in terms of this dg Lie
algebra. For example, the dg Lie algebra governing deformations of a
complex manifold $M$ is the Dolbeault complex $(\Omega^{-1,\bullet}
(M), \bar{\del})$ of the holomorphic tangent bundle of $M$. The
bracket, known as the \emph{Nijenhuis bracket}, is given by the
commutator of vector fields combined with the wedge product of
$(0,q)$-forms. The dg Lie algebra describing deformations of an
associative algebra $A$ is its Hochschild complex $C^\bullet (A,A)$
along with the Gerstenhaber bracket. The metatheorem is probably not
that much a conclusion, but rather a manifesto: If your favorite
deformation theory does not come from a dg Lie algebra, you should
make an effort to find one. This would bring a chance to bear fruit
for your theory.

I may give ``plausible reasoning''to convince the reader that the
metatheorem should hold for philosophical reasons. I am sure this
argument is somewhat of folklore, but I had never thought about it
until Jim Stasheff forwarded to me a question of Samir Shah as to
\emph{really why} the gods of mathematics had designed the metatheorem
to be true. Here is what I think. Deformation theory describes the
tangent cone at a point $x$ of the moduli space $M$ of the
problem. The tangent cone is $\Spec \Gr \OO_{M,x}$, where the
associated graded $\Gr$ is taken with respect to powers of the maximal
ideal of the local ring $\OO_{M,x}$, the stalk of the structure sheaf
$\OO_M$ at $x$. Usually, this cone is singular. You may resolve this
singularity within derived algebraic geometry, for instance, find a
free dg-commutative algebra whose cohomology is $\Gr \OO_{M,x}$. A free
dg-commutative algebra is equivalent to an $\Li$-algebra,
\emph{cf}.\ a remark before Theorem \ref{Q} below. Then you take a
quasi-isomorphic dg-Lie algebra, and you are done.

\subsection{The Maurer-Cartan Equation \& deformation functor}

Given a dg Lie algebra $\g = \bigoplus_{n \in \nz} \g^n$ with a
differential $d$ of degree $\abs{d} = 1$, the \emph{Maurer-Cartan set}
\[
\MC_{\g} := \{ S \in \g^1 \; | \; dS + \frac{1}{2} [S,S] = 0 \}
\]
is the set of degree-one solutions $S$, called \emph{Maurer-Cartan
  elements}, of the \emph{Maurer-Cartan Equation $($MCE$\,)$}
\begin{equation}
\label{mce}
dS + \frac{1}{2} [S,S] = 0,
\end{equation}
also known as the \emph{CME}, the \emph{Classical Master Equation}.

A dg Lie algebra $\g$ defines a much richer object, called a
\emph{deformation functor}:
\begin{align*}
\CLAlg & \to  \Set,\\
(R,\m) & \mapsto  \MC_\g(R),
\end{align*}
where $\CLAlg$ is the \emph{category of complete local $k$-algebras of
  finite type} and $(R,\m)$ is an object of it, $\Set$ is the category of
sets, and
\begin{equation}
\label{mcf}
\MC_\g(R) := \{ S \in (\g \; \widehat{\otimes} \; \m)^1 \; | \; dS + \frac{1}{2} [S,S]
= 0 \}.
\end{equation}
This set is interpreted as the set of deformations over $\Spec R$ of
the mathematical object whose deformation theory is governed by
$\g$. For example, when $\g$ is the Hochschild complex of an
associative algebra $A$, the set $\MC_g(R)$ is the set of associative
$R$-linear multiplications on $A \; \widehat{\otimes} \; R$ extending
the original multiplication on $A$.

\begin{example}
Let $\h$ be a Lie algebra. Then the dg Lie algebra of based, graded
coderivations
\[
\g := \Coder_*(S(\h[1])) = \hom_k (S^{>0}(\h[1]), \h[1])
\]
of the cofree conilpotent cocommutative coalgebra $S(\h[1])$ describes
the deformation theory of the Lie algebra $\h$. The differential on
$\g$ is defined as follows:
\[
d :=
\begin{cases}
-[-,-]: S^2(\h[1]) \to \h[1], \text{ the Lie bracket on $\h$,} &
\text{for $n = 2$},\\
0 & \text{for all other $n$}.
\end{cases}
\]
The funny sign is a matter of convention, which becomes useful when
one generalizes this theory to the case when $\h$ is an
$\Li$-algebra. It is the matter of a straightforward checkup that the
condition that $d$ is a codifferential, $d^2 = 0$, is equivalent to
the Jacobi identity for the Lie bracket $[-,-]$. A \emph{deformation
  of $\h$ over a complete local algebra} $(R,\m)$ is, by definition, a
new bracket $[-,-]'$ on $\h \; \widehat{\otimes} \; R$ which reduces
to the original bracket $[-,-]$ on $\h$ modulo $\m$. Thus, we have
\[
[x,y]' = [x,y] + S(x,y) \qquad \text{for each $x,y \in \h$}
\]
and some $S(x,y) \in \Hom_k (\Lambda^2(\h),\h) \; \widehat{\otimes} \;
\m = \hom_k (S^{>0}(\h[1]), \h[1])^1 \; \widehat{\otimes} \; \m$ such
that $[-,-]'$ satisfies the Jacobi identity. As above, this new
bracket produces a new, deformed codifferential on $\g \;
\widehat{\otimes} \; R$:
\[
d' = d + [S,-],
\]
for which the equation $(d')^2 = 0$ is equivalent to the Jacobi
identity for $[-,-]'$. Now observe that $(d')^2 = dS +
\frac{1}{2}[S,S]$. This implies that a deformation of $\h$ over $R$ is
equivalent to the choice of a Maurer-Cartan element $S \in \MC_\g(R)$.
\end{example}

\subsection{Representability theorems}

Considering the opposite \emph{category $\FLAff := \CLAlg^{\op}$ of
  formal local affine $k$-schemes of finite type} whose object
corresponding to an algebra $R$ is denoted by $\Spec R$, we may turn
the deformation functor into a contravariant one:
\begin{align*}
\FLAff\,^{\op} & \to  \Set,\\
\Spec R & \mapsto  \MC_\g(R),
\end{align*}
and speak of its representability, possibly in a larger category.

Note that a dg Lie algebra $\g$ defines a pointed formal dg-manifold
$\g[1]$ determined by the symmetric coalgebra $S(\g[1])$ with the
codifferential induced by the linear map $l: S(\g[1]) \to \g[1]$ whose
restriction $l_n$ to $S^n(\g[1])$ is defined by the following formula:
\[
l_n :=
\begin{cases}
d: \g[1] \to \g[1], \text{ the differential on $\g$,} & \text{for
  $n = 1$},\\
\pm [-,-]: S^2(\g[1]) \to \g[1], \text{ the Lie bracket on $\g$,} &
\text{for $n = 2$},\\
0 & \text{for all other $n$}.
\end{cases}
\]
The sign for $n=2$ is given by $l_2(x,y) = (-1)^{\abs{x}} [x,y]$ for
$x, y \in \g[1]$.

\begin{remark}
It is also useful to recall at this point that the structure of a
pointed formal dg-manifold on the pointed formal graded manifold
$\g[1]$ associated to a graded vector space $\g$ is equivalent to the
structure of an $\Li$-\emph{algebra} on $\g$. Moreover, an
$\Li$-\emph{morphism} $\g' \to \g$ between two $\Li$-algebras is by
definition a morphism of pointed formal manifolds $\g'[1] \to \g[1]$,
which is, by definition, nothing but a morphism of coaugmented
dg-coalgebras $S(\g'[1]) \to S(\g[1])$. Since $S(\g[1])$ is cofree,
every such morphism is determined by a linear map $S^{>0}(\g'[1]) \to
\g[1]$, projection of the morphism to the cogenerating space
$\g[1]$. Thus, the \emph{category $\LiAlg$ of $\Li$-algebras with
  $\Li$-morphisms} becomes a full subcategory of the \emph{category
  $\PFDGMan$ of pointed formal dg-manifolds}, namely the full
subcategory of \emph{linear} pointed formal dg-manifolds.
\end{remark}

\begin{theorem}[Quillen, as per \cite{kontsevich-soibelman}]
\label{Q}
The deformation functor $\MC_\g$ is represented by the pointed formal
dg-manifold $\g[1]$, \emph{i.e}., there is a natural isomorphism
\[
\begin{split}
\MC_\g(R) & \xrightarrow{\sim} \Mor_{\PFDGMan} (\Spec R, \g[1])\\
   & := \Hom_{\CDGCAlg} ( (R^*, 0), (S(\g[1]), D)).
\end{split}
\]
\end{theorem}
Here $\Spec R$ is regarded as a pointed formal dg-manifold with a zero
differential, as in the Conventions section of the introduction,
$\Mor_{\PFDGMan}$ stands for the set of morphisms of pointed formal
dg-manifolds, and $\Hom_{\CDGCAlg}$ for the set of homomorphisms of
\emph{coaugmented differential graded coalgebras}.

\begin{remark}
  We do not consider solutions of the Maurer-Cartan equation \emph{up
    to homotopy}, or \emph{gauge equivalence} classes of solutions,
  here and in the sequel (for the Quantum Master Equation) for a
  reason. We can always extend the scalars and tensor the given dg-Lie
  or $\Li$-algebra $\g$ with the dg-algebra of polynomial differential
  forms on the $n$-simplex $\Delta^n$:
  $\g \otimes \Omega^\bullet (\Delta^n)$. If we do this for each
  $n \ge 0$, we will obtain a simplicial dg-Lie algebra. Solutions of
  the Maurer-Cartan equation in this simplicial dg-Lie algebra will
  form a deformation functor with values in simplicial sets, whose
  topology will reflect homotopy-theoretic properties of the
  deformation functor. For example, the set $\pi_0$ of its path
  components will be the functor of gauge equivalence classes of
  solutions. Thus, a mere extension of the deformation functor
  $\MC_\g$ to dg-commutative algebras will recover necessary
  homotopy-theoretic information.
\end{remark}

\begin{proof}
Ignore the differentials for the time being. Since $S(\g[1])$ is
cofree, homomorphisms $R^* \to S(\g[1])$ of conilpotent coaugmented
coalgebras are in a natural bijection with homogeneous $k$-linear maps
$\m^* \to \g[1]$, which are in bijection with the space $\g^1 \;
\widehat{\otimes} \; \m = (\g \; \widehat{\otimes} \; \m)^1$, because
of our finiteness assumption for $(R,\m)$. Explicitly, a coalgebra
homomorphism corresponding to an element $S \in (\g \;
\widehat{\otimes} \; \m)^1 = (\g[1] \; \widehat{\otimes} \; \m)^0 =
\Hom_k (\m^*, \g[1])$ is $\exp (S) \in \Hom_k (R^*, S(\g[1]))$, where
the exponential is taken in the sense of the convolution product on
the space of linear maps from a cocommutative coalgebra to a
commutative algebra.

Now recall that the homomorphism $\exp(S)$ must respect the
differentials. In this case, this means $D \circ \exp(S) =
0$. However, $D \circ \exp(S)$, being a coderivation of $R^*$ with
values in the cofree conilpotent coalgebra $S(\g[1])$ over the
homomorphism $\exp(S)$, is determined by its projection
\[
l_1(S) + \frac{1}{2!} l_2 (S,S) = dS + \frac{1}{2} [S,S]
\]
to the space $\g[1]$ of cogenerators.
\end{proof}

The unsettling discrepancy between the category on which the
Maurer-Cartan functor $\MC_\g$ is defined and the category in which it
is ``represented'' may nicely be resolved by the following tune-up.

\begin{theorem}[Chuang-Lazarev \cite{chuang-lazarev}]
\label{Chuang-Lazarev}
For any dg Lie algebra $\g$, the contravariant functor $\MC_\g$, as in
\eqref{mcf}, extended from $\FLAff$ to the category $\LiAlg$ of
$\Li$-algebras:
\begin{gather*}
\MC_\g: \LiAlg^{\op} \to \Set,\\
\MC_\g(\g') := \{ S \in \hom^1_k (S^{>0} (\g'[1]), \g) \; | \; DS +
\frac{1}{2} [S,S] = 0 \},
\end{gather*}
where $D$ is the standard differential on $\hom_k$ combining the
differentials on $\g$ and $S(\g'[1])$ and the bracket combines the
bracket on $\g$ with the coproduct on $S(\g'[1])$, is represented by
the dg Lie algebra $\g$ itself, regarded as an $\Li$-algebra. In other
words, there is a natural isomorphism
\[
\begin{split}
\MC_\g (\g') & \xrightarrow{\sim} \Mor_{\LiAlg} (\g', \g)\\
 & := \Mor_{\PFDGMan} (\g'[1], \g[1]) := \Hom_{\CDGCAlg} (S(\g'[1]), S(\g[1])).
\end{split}
\]
\end{theorem}

\begin{proof}
The proof of Theorem~\ref{Q} works ditto in this case.
\end{proof}

\begin{remark}
This theorem admits even finer tuning, in which the dg Lie algebra
$\g$ is replaced with an $\Li$-algebra and the MCE \eqref{mce} is
replaced with an \emph{Extended Maurer-Cartan Equation $($EMCE$)$}:
\[
DS + \frac{1}{2!} [S,S] + \frac{1}{3!} [S,S,S] + \dots = 0,
\]
which may equivalently be written as
\[
\sum_{n=1}^\infty \frac{1}{n!} l_n (S, \dots, S) = 0,
\]
where $l_1 := D$ is the differential and $l_n := [-, \dots, -]$, $n
\ge 2$, are the higher, $\Li$ brackets on $\g$ (with ``scalars
extended'' to $\hom_k (S(\g'[1]), -)$). Everything else in the wording
of the theorem remains intact.
\end{remark}

\section{Quantum theory: QME \& quantum deformation theory}

\subsection{Overture}
\label{overture}

Let me start this section with a probably more contentious metatheorem
than the previous one.

\begin{metatheorem}
Every reasonable quantum deformation problem comes from a
$\BVi$-algebra.
\end{metatheorem}

The classical version, Metatheorem~\ref{meta-cl}, can be stated
equivalently in a similar form, as follows.

\begin{metatheorem-quoted}{$\mathbf{\ref{meta-cl}'}$}
Every reasonable classical deformation problem comes from an
$\Li$-algebra.
\end{metatheorem-quoted}

Indeed, on the one hand, any dg-Lie algebra is an $\Li$-algebra, and
on the other hand, if we have managed to construct an $\Li$-algebra
governing our deformation problem, then a quasi-isomorphic dg-Lie
algebra will describe this deformation problem equally well. A related
statement about quasi-isomorphic dg-BV-algebras and solutions of the
Quantum Master Equation \eqref{QME} below is proven by K.~Costello,
\cite[Section 5]{costello:tft}.

In the rest of the paper, I would like to provide evidence for the
quantum metatheorem. Before doing that, I need to define a few
things. Roughly speaking, under a \emph{quantum deformation problem} I
understand a deformation problem for a structure based on graphs
rather than trees or engaging higher genera rather than genus
zero. Those include structures of algebras over PROPs and modular
operads, rather than over operads and dioperads. Examples of such
could be Frobenius algebras, any types of bialgebras and their
homotopy versions, such as $\Li$-bialgebras and
$\IBLi$-algebras. Perhaps, deformations of stable maps in
Gromov-Witten theory may also be classified as quantum deformations. I
will return to examples later, after discussing the appropriate
setup. So, what is a $\BVi$-algebra?

\subsection{Differential graded BV- and $\BVi$-algebras}

As in the classical case, typical quantum deformation problems will be
coming from dg-BV-algebras, rather than $\BVi$ ones. Thus, let us
first discuss dg-BV-algebras.

\begin{definition}
\label{BV}
A \emph{dg-BV-algebra} is a dg-commutative associative algebra $(V,d)$
with a second-order differential $\Delta$. By a \emph{second-order
  differential} on a dg-com\-mu\-ta\-tive algebra $(V,d)$ we mean a
linear operator $\Delta: V \to V$ of degree $-1$, called a \emph{BV
  operator}, (graded) commuting with the differential $d$: $[\Delta,
  d] = 0$, annihilating the constants: $\Delta(1) = 0$, squaring to
zero: $\Delta^2 = 0$, and being a \emph{differential operator of
  second order}, which is a shortcut for \emph{order $\le 2$}:
\begin{equation*}
%\label{Grothendieck}
[[[\Delta, L_a],L_b],L_c] = 0 \qquad \text{for any $a,b,c \in V$,}
\end{equation*}
where $L_a: x \mapsto ax$ is the operator of left multiplication by
$a$ on $V$. For the purpose of this note, we will also assume a rather
nonstandard piece of structure, that of a conilpotent graded
cocommutative coalgebra on $V$. We will impose minimal compatibility
between the two structures, namely, that the unit of the algebra
structure is a coaugmentation of the coalgebra structure and that the
counit of the coalgebra structure is an augmentation of the algebra
structure. We will also assume that the differentials are compatible
with the augmentation homomorphism $V \to k$, where both $d$ and
$\Delta$ act trivially on $k$.
\end{definition}

\begin{remark}
Even though the traditional definition of a dg-BV-algebra does not
assume any coalgebra structure, imposing it is not unprecedented: it
was secretly used in \cite{cieliebak-latschev,
  cieliebak-fukaya-latschev} in the study of $\BVi$- and
$\IBLi$-morphisms. Our work \cite{mv} with Markl arose from our
discovering this secret and attempting to leak this information to the
public. The coalgebra requirement is rather mildly restrictive: every
augmented dg-commutative algebra carries a \emph{trivial} conilpotent
cocommutative comultiplication defined by $\delta(1) : = 1 \otimes 1$,
$\delta(a) := a \otimes 1 + 1 \otimes a$ for $a$ in the augmentation
ideal, see, \emph{e.g}., \cite{mv}. If we do not mention a specific
comultiplication in the sequel, we will assume the trivial
comultiplication.
\end{remark}

Note that the failure of $\Delta$ to be a derivation is measured by a
Lie bracket of degree $-1$, often called an \emph{antibracket}:
\begin{equation}
\label{der-bracket}
\begin{split}
\{a,b\} & := (-1)^{\abs{a}} (\Delta(ab) - (\Delta a) b - (-1)^{\abs{a}}
a (\Delta b))\\
& = (-1)^{\abs{a}} [[\Delta, L_a], L_b] (1) \qquad \text{for $a, b \in V$,}
\end{split}
\end{equation}
which turns $V$ into a dg-Gerstenhaber algebra.

\begin{example}[The Chevalley-Eilenberg complex of a dg-Lie algebra]
\label{CE}
Let $\g$ be a dg-Lie algebra. % Consider the graded symmetric algebra
% $S(\g[-1])$ on the shifted graded vector space $\g[-1]$. Note that
% $S(\g[-1])$ with the standard shuffle coproduct becomes a graded
% commutative, associative, cocommutative, coassociative bialgebra.
Then its \emph{Chevalley-Eilenberg $($CE$)$ complex} $C_{\bullet} (\g;
k) : = S(\g[-1])$ is a dg-BV-algebra. The differential $d$ is the
internal differential on $S(\g[-1])$, and the BV operator $\Delta$ is
the following part of the CE differential:
\[
\Delta (x_1 \dots x_n) := \sum_{i < j} (-1)^{\abs{x_1} + \dots +
  \abs{x_i} +\epsilon} x_1 \dots [x_i, x_j] \dots \widehat{x}_j \dots
x_n,
\]
where $x_1, \dots, x_n \in \g[1]$, $\abs{x}$ is the degree of $x$ in
$\g[-1]$, and $(-1)^\epsilon$ is the Koszul sign gotten from commuting
$x_1 \dots x_n$ to $x_1 \dots x_i x_j x_{i+1} \dots \widehat{x}_j
\dots x_n$ in $S(\g[-1])$. More generally, if $\g$ is an
$\Li$-algebra, then $S(\g[-1])$ with the CE differential becomes a
(commutative) $\BVi$-algebra, see
\cite{braun-lazarev,bashkirov-voronov1}.
\end{example}

\begin{example}[The Chevalley-Eilenberg complex of an involutive Lie
bialgebra, see \cite{sullivan-terilla-tradler, cieliebak-latschev,
  drummond-cole-terilla-tradler, cieliebak-fukaya-latschev}]
\label{IBL}
Let $\g$ be an \emph{involutive Lie bialgebra}, that is to say, a Lie
bialgebra $(\g, [-,-], \delta)$, $\delta: \g \to \g \wedge \g$ being
the cobracket, satisfying an \emph{involutivity condition}:
$ [-,-] \circ \delta = 0$.
% Let $R = k [[\lambda]]$ be the (complete local) algebra of formal
% power series.
%Set
%\begin{align*}
%l_{2,1} (x \otimes y) & := - [x,y],\\
%l_{1,2} (x) & := \delta(x)
%\end{align*}
%for all $x,y \in \g[-1]$.
Let $\Delta$ be the CE differential corresponding to the Lie algebra
structure on $\g$, as in the previous
example. %extension of $l_{2,1}$ to
%$S(\g[-1])$ as a second-order differential operator of degree $-1$
%satisfying $\Delta (1) = 0$. Explicitly, use \eqref{Grothendieck},
%which rewrites here as
%\[
%\Delta (abc) = \Delta (ab) c - \Delta (ac) b + \Delta (bc)a - \Delta
%(a)b c + \Delta (b) a c - \Delta (c) a b
%\]
%for $a, b, c \in S(\g[-1])$, to extend $l_{2,1}$ to $\Delta$
%recursively. Similarly,
Extend the cobracket $\delta$ as a degree-one derivation $d$ of
$S(\g[-1])$. Then $(S(\g[-1]), d, \Delta)$ is a dg-BV-algebra. Note
that without the involutivity condition, the BV operator $\Delta$ will
no longer commute with the differential $d$, but if we forget
$\Delta$, $S(\g[-1])$ will still be a dg-Gerstenhaber algebra, see
\cite{ks:exact}. See ibid.\ for a construction of a different BV
operator on the dg-Gerstenhaber algebra $S(\g[-1])$ in the case when
$\dim \g < \infty$ and the dual $\g^*$ carries a Lie-bialgebra
structure which is triangular, rather than involutive.
\end{example}

\begin{example}[The Chevalley-Eilenberg complex of a bi-dg-Lie algebra]
\label{bi-dg}
This example is very important for quantum deformation theory, see
Examples \ref{twisted-End} and \ref{twisted-O}. Let $\g$ be a graded
Lie algebra with two commuting differentials: $d$ of degree 1 and
$\Delta$ of degree $-1$. We may call such $\g$ a \emph{bi-dg-Lie
  algebra}. Consider the graded symmetric algebra $S(\g[-1])$ on the
shifted graded vector space $\g[-1]$. Extend the differential $d$ to
$S(\g[-1])$ as an ``internal'' differential with respect to
multiplication. Extend the Lie bracket $[-,-]$ as a degree $-1$
biderivation to an antibracket $\{-,-\}$ on $S(\g[-1])$, known as the
\emph{Schouten bracket}. Then extend $\Delta$ to a second-order
differential operator on $S(\g[-1])$ by the formula
\begin{equation*}
\Delta(ab) = (\Delta a) b + (-1)^{\abs{a}} a (\Delta b) +
(-1)^{\abs{a}} \{a,b\} \qquad \text{for $a, b \in S(\g[-1])$.}
\end{equation*}
The resulting triple $(S(\g[-1]), d, \Delta)$ is a dg-BV-algebra. By
the way, for every dg-BV-algebra $V$, the shifted space $V[1]$ carries
a bi-dg-Lie algebra structure with respect to the
antibracket. Moreover, for the dg-BV-algebra $S(\g[-1])$ of this
example, the natural inclusion $\g = S^1 (\g[-1])[1] \hookrightarrow
S(\g[-1])[1]$ is a morphism of bi-dg-Lie algebras.
\end{example}

\begin{example}[The bar complex of an associative algebra, see
Terilla-Tradler-Wilson \cite{ttw}]
\label{TTW}

Let $A$ be a dg-associative algebra and $T(A[-1])$ be the dg-tensor
coalgebra on the shifted dg-vector space $A[-1]$. Then $T(A[-1])$ with
the shuffle product and the BV operator
\[
\Delta (a_1 \otimes \dots \otimes a_n) := \sum_{i} (-1)^{\abs{a_1} +
  \dots + \abs{a_i}} a_1 \otimes \dots \otimes (a_i \cdot a_{i+1})
\otimes \dots \otimes a_n
\]
for $a_i \in A[-1]$, becomes a dg-BV-algebra. Note that we need to
choose a conilpotent cocommutative coproduct on $T(A[-1])$, such as
the shuffle coproduct, or the trivial coproduct described in the
remark after Definition \ref{BV}, to fit our definition of a
dg-BV-algebra. I suspect that $T(A[-1])$, perhaps with the original
coassociative coproduct, is responsible for quantum deformation theory
over noncommutative, associative parameter rings.
\end{example}

\begin{example}[The bar complex of an $\OO$-algebra]
Let $\OO$ be a Koszul quadratic operad of vector spaces, $\OO^!$ be
its Koszul-dual operad, and $V$ be a dg-$\OO$-algebra. Then the cofree
conilpotent $\OO^!$-coalgebra $F^c_{\OO^!} (V[1])$ acquires a
codifferential $d + \Delta$, where $d$ is the internal differential
and $\Delta$ is the coderivation corresponding to the $\OO$-algebra
structure on $V$. Based on the particular cases of $\OO$ being the Lie
operad or the associative operad, as in Examples \ref{CE} and
\ref{TTW}, respectively, I anticipate that $F^c_{\OO^!} (V[-1])$ with
the differential $d$ and the BV operator $\Delta$ will be a dg-BV
algebra and give rise to deformation theory with $\OO^!$-algebras as
parameter rings. The 2018 honors thesis \cite{lucy} of Lucy Yang at
the University of Minnesota aims to prove that $\Delta$ is indeed a BV
operator.
\end{example}

\begin{example}[Functions on an odd symplectic manifold, see Schwarz
\cite{schwarz,getzler:BV}]
\label{schwarz}

Let $M$ be an odd symplectic supermanifold of dimension $(n|n)$ with a
volume form. Then the algebra $C^\infty(M)$ of smooth functions on $M$
becomes a BV-algebra with the BV operator given by
\[
\Delta := \sum_i \frac{\del^2}{\del x_i \del \xi_i}
\]
in super Darboux coordinates $(x_1, \dots , x_n | \; \xi_1, \dots,
\xi_n)$. Here we do not assume much of a dg structure, \emph{i.e}.,
the grading is actually a $\nz/2\nz$-grading and the differential $d$
is just zero. Also, we do not assume any comultiplication in this
example. However, to get one, one can choose a basepoint on $M$ and
use the associated augmentation on $C^\infty(M)$ to define a trivial
coproduct, as in the Remark after Definition \ref{BV}. Still, in
general there will be no compatibility between the augmentation and
$\Delta$, which we require of our dg-BV-algebras. Thus, this example,
albeit fundamental, is an outlier in our context.

A particular case of this example, which I present in the graded,
rather than $\nz/2\nz$-graded version, is the example of $T^*[1]M$, a
shifted cotangent bundle to an oriented $n$-manifold $M$, see
\cite{campos}. Functions on this graded manifold are nothing but
multivector fields $\Gamma(M, S(T[-1]M))$. The volume form on $M$
gives an isomorphism: $f: \Gamma (M, S^p (T[-1]M)) \to \Omega^{n-p}
(M)$. Then $\Delta = f^{-1} \circ d_{\operatorname{dR}} \circ f$
defines the structure of a BV-algebra on $\Gamma(M, S(T[-1]M))$,
\emph{i.e}., a dg-BV-algebra with a zero differential.
\end{example}

\begin{definition}
A \emph{$($commutative$)$ $\BVi$-algebra} is a graded commutative
associative algebra $V$ with a sequence of differential operators
$\Delta_n$ of order $\le n$ with $n \ge 1$. A \emph{differential
  operator of order $\le n$} on a graded commutative algebra $V$ is a
linear operator $\Delta_n: V \to V$ satisfying
\begin{equation*}
%\label{Grothendieck}
[\dots [[\Delta_n, L_{v_0}],L_{v_1}], \dots, L_{v_n}] = 0 \qquad
\text{for any $v_0, v_1, \dots, v_n \in V$.}
\end{equation*}
We also require that the $k[[\hbar]]$-linear operator $\dhat := \sum_{n
  \ge 1} \hbar^{n-1} \Delta_n: V[[\hbar]] \to V[[\hbar]]$, which we
call a $\BVi$ \emph{operator}, where $\hbar$ is a formal variable with
$\abs{\hbar} = 2$, be of total degree $1$, kill constants: $\dhat (1) =
0$, and square to zero: $\dhat^2 = 0$. Another requirement, specific to
this paper, is that $V$ has a a conilpotent graded cocommutative
coalgebra structure, mildly compatible with the algebra structure, as
in Definition \ref{BV}: the unit of the algebra structure is a
coaugmentation of the coalgebra structure, and the counit of the
coalgebra structure is an augmentation of the algebra structure. And
we again assume that $\dhat$ is compatible with the augmentation
homomorphism $V[[\hbar]] \to k[[\hbar]]$, where $\dhat$ acts trivially
on $k[[\hbar]]$.
\end{definition}

\begin{example}
Every dg-BV-algebra $(V,d,\Delta)$ is a $\BVi$-algebra with
$\dhat = d + \hbar \Delta$ or $\Delta_1 = d$, $\Delta_2 =
\Delta$, and $\Delta_n = 0$ for all other $n$.
\end{example}

The $\BVi$ operator $\dhat$ generates a whole family of ``derived''
antibrackets
\begin{equation*}
\begin{split}
\{v_1, v_2, \dots, v_n\} & := \frac{1}{\hbar^{n-1}} [\dots [[\dhat,
      L_{v_1}],L_{v_2}], \dots, L_{v_n}] (1)\\ & = [\dots [[\Delta_n
      +\hbar \Delta_{n+1} + \hbar^2 \Delta_{n+2} + \dots, L_{v_1}],L_{v_2}], \dots, L_{v_n}]
(1)
\end{split}
\end{equation*}
for $v_1, v_2, \dots, v_n \in V$. Note a conventional sign difference
with \eqref{der-bracket} for $n=2$. These antibrackets (after being
multiplied back by $\hbar^{n-1}$, to be precise) define the structure
of an $\Li$-algebra on $V$, compatible with the product on $V$ in the
way that the failure of the $n$th antibracket to be a multiderivation
is measured by the $(n+1)$st antibracket. These statements are an
original result of J.~Alfaro, I.~A. Batalin, K.~Bering, P.~H. Damgaard
and R.~Marnelius
\cite{alfaro-damgaard,batalin-bering-damgaard,batalin-marnelius:q},
see also F.~Akman \cite{akman:bv,akman:master}, Th.~Th. Voronov
\cite{tvoronov:hdb,tvoronov:hdba}, and D.~Bashkirov and the author
\cite{bashkirov-voronov1}.

Commutative $\BVi$-algebras appeared in \cite{cieliebak-latschev} in
the study of Symplectic Field Theory. We will be dropping the
adjective ``commutative,'' despite the fact that our commutative
$\BVi$-algebras do not fit the definition of an $\OO_\infty$-algebra in
the sense of being an algebra over a cofibrant model $\OO_\infty$ of an
operad $\OO$. The correct, $C_\infty$ version of the notion of a
$\BVi$-algebra and its relation to the notion of a commutative
$\BVi$-algebra is described in \cite{galvez-carrillo-tonks-vallette}.

\subsection{QME \& quantum deformation functor}

Again, let $\hbar$ be a formal variable of degree 2:
\[
\abs{\hbar} = 2.
\]
First of all, a dg-BV-algebra $V$ will be related to quantum
deformations through the corresponding \emph{quantum deformation
  functor}
\begin{align*}
\CLAlg & \to  \Set,\\
(R,\m) & \mapsto  \QM_V (R),
\end{align*}
which associates to a complete local algebra $(R,\m)$ the set
\[
\QM_V (R) := \{ S \in V[[\hbar]]^2 \, \widehat{\otimes}\; \m \; | \;
dS + \hbar \Delta S + \frac{1}{2} \{S,S\} = 0 \}
\]
of solutions of the \emph{Quantum Master Equation $(\!$QME$\,)$}:
\begin{equation}
\label{QME}
dS + \hbar \Delta S + \frac{1}{2} \{S,S\} = 0,
\end{equation}
which is equivalent to
\[
\widehat d\, e^{S/\hbar} = 0,
\]
where
\[
\dhat := d + \hbar \Delta,
\]
in the space $V((\hbar)) \; \widehat{\otimes} \; R$ of formal
Laurent series, because of the following remarkable formula
\begin{equation}
\label{big-formula}
e^{-S/\hbar} \circ \widehat d \circ e^{S/\hbar} = \widehat d + \{S,
-\} + \frac{1}{\hbar} \left( \widehat d S + \frac{1}{2}
\{S,S\}\right),
\end{equation}
for operators on $V((\hbar)) \; \widehat{\otimes} \; R$, where we
abuse notation by writing elements, such as $e^{S/\hbar}$, in lieu of
the operators, such as $L_{e^{S/\hbar}}$, of left multiplication by
these elements. This formula follows from a celebrated, and much more
compact, identity
\[
\Ad_{e^A} = e^{\ad_A}
\]
for linear operators on a Lie algebra, in our case evaluated on
$\widehat d$ with $A = L_{-S/\hbar}$. If we apply the operators on
both sides of \eqref{big-formula} to the unit element 1, we obtain an
equation
\[
e^{-S/\hbar} \widehat d \, (e^{S/\hbar}) = \frac{1}{\hbar}
\left(\widehat d S + \frac{1}{2} \{S,S\} \right)
\]
on elements in $V((\hbar)) \; \widehat{\otimes} \; R$, which yields
the equivalence of the two forms of the QME above.

The quantum deformation functor admits a generalization to a
$\BVi$-algebra $(V, \dhat)$. Here are the adjustments one needs to make
in this case. The quantum deformation functor $\QM_V: \CLAlg \to \Set$
associates to a complete local algebra $(R,\m)$ the set
\begin{equation}
\label{QM}
\QM_V (R) := \{ S \in V[[\hbar]]^2 \, \widehat{\otimes}\; \m \; | \;
\dhat \, e^{S/\hbar} = 0 \}
\end{equation}
of solutions of the \emph{Quantum Master Equation $(\!$QME$\,)$}:
\begin{equation}
\label{QMEi}
\dhat \, e^{S/\hbar} = 0,
\end{equation}
which is equivalent to
\[
\dhat S + \frac{1}{2!} \{S,S\} + \frac{1}{3!} \{S,S,S\} + \dots = 0.
\]
Again, the equivalence follows from a generalization of
\eqref{big-formula}:
\begin{multline*}
e^{-S/\hbar} \circ \dhat \circ e^{S/\hbar}
= \dhat + \{S, -\} + \frac{1}{2!}\{S, S, -\} + \dots\\
+ \frac{1}{\hbar} \left( \dhat S + \frac{1}{2!} \{S,S\} + \frac{1}{3!}
\{S,S,S\} + \dots \right),
\end{multline*}
proven exactly in the same way as in the dg-BV case.%, see for example
%\cite{braun-lazarev,bashkirov-voronov1,mv}.

Solutions to the QME in a dg-BV- or $\BVi$-algebra, in general, may be
considered as distinguished deformations of the BV or $\BVi$
operator. In particular cases, they may describe interesting
structures. For example, for the dg-BV-algebra $\Gamma(M, S(T[-1]M))$
of multivector fields on an oriented manifold $M$, see the end of
Example \ref{schwarz}, a solution of the QME is a linear polynomial
$S_0 + \hbar S_1$, where $S_0$ is a bivector field and $S_1$ is just a
function satisfying the relations $[S_0, S_0] = 0$ and $\Delta(S_0) +
[S_1, S_0] = 0$. As Campos pointed out to me, these data define a
unimodular Poisson structure on $M$, see \cite{willwacher-calaque}.

\subsection{Quantum representability theorem}

To talk about representable functors in the QME context, we need to
switch to different categories, those of $\BVi$-spaces and
$\BVi$-algebras.

\begin{definition}
\label{BVi}
Let $(V, \dhat)$ and $(V', \dhat')$ be two $\BVi$-algebras. A
\emph{$\BVi$-morphism} $V \to V'$ is a $k$-linear map $\varphi: V \to
V'[[\hbar]]$ of degree two such that
\begin{enumerate}
\item $\varphi(1) = 0$;
\item $\dhat' \circ \exp(\varphi/\hbar) = \exp (\varphi/\hbar) \circ
  \dhat$, where the exponential $\exp$ is taken with respect to the
  convolution product on $\Hom_{k}(V, V'((\hbar)))$;
\item $\varphi = \varphi_0 + \hbar \varphi_1 + \hbar^2 \varphi_2 +
  \dots$, where $\varphi_n: V \to V'$ is a \emph{differential operator
    of order $\le n+1$ over the trivial algebra homomorphism} $V \to
  V'$, which takes the augmentation ideal $\m$ of $V$ to zero,
  \emph{i.e}., $\varphi_n (\m^{n+2}) = 0$.
\end{enumerate}
\end{definition}

This definition is somewhat more general than the original one by
Cieliebak and Latschev \cite{cieliebak-latschev} (or Cieliebak,
Fukaya, and Latschev \cite{cieliebak-fukaya-latschev}): if we require
our $\varphi_n$ to be a differential operator of order $\le n$ over
the trivial algebra homomorphism for each $n \ge 0$, then
$\varphi/\hbar$ will be a $\BVi$-morphism in their sense. The
exponential makes sense, because of Condition (1) and the conilpotency
of the coproduct on $V$. Composition of $\BVi$-morphisms $\varphi$ and
$\psi$ is done by composing $\exp(\varphi/\hbar)$ and
$\exp(\psi/\hbar)$. The fact that composition of exponentials is the
exponential of a $\BVi$-morphism follows from the existence of the
logarithm and its extension to Laurent series in $\hbar$, see
\cite{bashkirov-voronov1,mv}, and checking that the series $\hbar \log
(\exp(\varphi/\hbar) \circ \exp(\psi/\hbar))$ satisfies Properties
(1)-(3) of Definition \ref{BVi}. For example, to see that the series
does not contain any negative powers of $\hbar$, one verifies that
$\lim_{\hbar \to 0} \hbar \log (\exp(\varphi/\hbar) \circ
\exp(\psi/\hbar))$ is finite.

Let $\BViAlg$ denote the category of $\BVi$-algebras and $\BViSp$ the
same category, interpreted geometrically: if $(V, \dhat)$ is a
$\BVi$-algebra, $\Spec V^*$ will denote the corresponding geometric
object, which we call a $\BVi$-\emph{space}. The idea is that this is
a geometric object, generalized functions, or distributions, on which
form the $\BVi$-algebra $V$.

Observe that the opposite category of complete local algebras forms a
subcategory of the category of $\BVi$-algebras:
\[
\CLAlg \subseteq \BViAlg^{\op}.
\]
To see this, observe that if $(R,\m)$ is a complete local $k$-algebra,
then its $k$-linear dual $R^*$ with the $\BVi$ operator $\dhat = 0$ and
multiplication defined to be zero on $\m^*$ is a $\BVi$-algebra. A
homomorphism $f: (R, \m_R) \to (S, \m_S)$ of complete local algebras
induces a dual morphism $f^*: S^* \to R^*$ of coalgebras. Then
$\varphi := \hbar \log f^*$ is a $\BVi$-morphism $S^* \to
R^*$. Indeed, let $\sfe$ be the unit of $\Hom_k (S^*, R^*)$ under the
convolution product. It is given by composing the unit morphism $k \to
R^*$ with the counit morphism $S^* \to k$. Then $\varphi = \hbar \log
f^* = \hbar(f^* - \sfe)$, because $\Ima (f^* - \sfe) \subseteq \m^*_R$
and $(\m^*_R)^2 = 0$. Therefore, $\varphi(1) = \hbar(1 - 1) = 0$. Also
$\log f^* = f^* - \sfe$ automatically vanishes on $(\m_S^*)^3 =
0$. Thus, it makes sense to talk about a functor $\CLAlg \to \Set$,
such as $\QM_V$, being represented by an object of the category
$\BViAlg$ or $\BViSp$.

\begin{theorem}
\label{first}
The quantum deformation functor $\QM_V$ associated to a dg-BV- or a
$\BVi$-algebra $V$ is represented by the $\BVi$-space $\Spec V^*$,
\emph{i.e}., there is a natural isomorphism
\[
\begin{split}
\QM_V (R) & \xrightarrow{\sim} \Mor_{\BViSp} (\Spec R, \Spec
V^*)\\ & := \Hom_{\BViAlg} ( (R^*, 0), (V, \dhat)).
\end{split}
\]
\end{theorem}

Park, Terilla, and Tradler in \cite{park-terilla-tradler} prove a
representability theorem of a rather different flavor for the quantum
deformation functor up to gauge equivalence. Our result is closer to
but does not directly follow from M\"unster-Sachs \cite[Section
  4.3]{muenster-sachs} or Markl-V \cite[Corollary 41]{mv}. However,
the proof, which we repeat here for completeness, is similar.

\begin{proof}
A solution $S \in V[[\hbar]]^2 \, \widehat{\otimes}\; \m$ of the QME
\eqref{QMEi} is by definition equivalent to a degree-two $k$-linear
map $S: \m^* \to V[[\hbar]]$ satisfying $\dhat \exp (S/\hbar) = 0$ or a
degree-two $k$-linear map $S: R^* \to V[[\hbar]]$ such that $S(1) =
0$. Each such $S$ automatically satisfies Property (3) of Definition
\ref{BVi}, because $(\m^*)^{n+2} = 0$ for all $n \ge 0$.
\end{proof}

A quantum analogue of Chuang-Lazarev's Theorem \ref{Chuang-Lazarev}
has a more natural wording and, naturally, a totally trivial
proof. The quantum deformation functor associated to a $\BVi$-algebra
$(V,\dhat)$ may be extended to a functor
\[
\QM_V: \BViAlg^{\op} \to \Set
\]
which takes a $\BVi$-algebra $(V', \dhat')$ to the set of
$\BVi$-morphisms $V' \to V$. One may view the equation $\dhat
\exp(\varphi/\hbar) = \exp(\varphi/\hbar) \dhat'$ as a QME on
$\varphi$. Then, tautologically, the functor $\QM_V$ is represented by
the $\BVi$-algebra $(V,\dhat)$ or the $\BVi$-space $\Spec V^*$.

On the other hand, the following less general representability theorem
may be more interesting.

Before wording the theorem, note that the category $\LiAlg$ of
$\Li$-algebras (and thereby the equivalent category $\PFDGMan$ of
pointed formal dg-manifolds) is a subcategory of the category
$\BViAlg$ of $\BVi$-algebras. Indeed, if $\g$ is an $\Li$-algebra,
then $S(\g[-1])$ is a $\BVi$-algebra, see Example \ref{CE}. An
$\Li$-morphism $\g \to \h$ is, by definition, a morphism of
coaugmented dg-coalgebras $S(\g[1]) \to S(\h[1])$, which is determined
by its components $\varphi_n: S^n(\g[1]) \to \h[1]$, $n \ge 1$, and $\varphi
= \sum_{n \ge 1} \hbar^n \varphi_n$ defines a $\BVi$-morphism $S(\g[-1])
\to S(\h[-1])$, see \cite[Theorem 4.8]{bashkirov-voronov1}.

Now define a version of the quantum deformation functor on the
opposite category $\LiAlg^{\op}$ of the category of
$\Li$-algebras. Let $(V, \dhat_V)$ be a $\BVi$-algebra and $\g$ an
$\Li$-algebra. Then $S(\g[1])$ is a coaugmented conilpotent
cocommutative dg-coalgebra with the codifferential $D_1 + D_2 + \dots$
defining the $\Li$ structure on $\g$: $D_n$ extends the $n$th bracket
$l_n: S^n(\g[1])\to \g[1]$ to a degree-one coderivation of
$S(\g[1])$. Likewise, $S(\g[-1])$ is a coaugmented conilpotent
cocommutative graded coalgebra with the codifferential $\dhat_\g := D_1
+\hbar D_2 + \hbar^2 D_3 + \dots$ on $S(\g[-1])[[\hbar]]$. Hence, the
graded vector space $\hom_k (S(\g[-1]), V)$ becomes a $\BVi$-algebra
with respect to the convolution product and the $\BVi$ operator $\Dh
(\Phi) := \dhat_V \circ \Phi - (-1)^{\abs{\Phi}}\Phi \circ
\dhat_\g$.\footnote{For $\Dh$ to define a $\BVi$ operator, it is
  essential that $\dhat_\g$ be a coderivation.} Thus, we can define the
value of the quantum deformation functor associated to $V$ on the
$\Li$-algebra $\g$ as the set 
\begin{equation}
\label{QMV}
\QM_V(\g) := \{ S = \sum_{n \ge 0} \hbar^n S_n \; | \; \Dh \,
e^{S/\hbar} = 0 \}
\end{equation}
of solutions to the \emph{QME for the $\BVi$-algebra} $\hom_k
(S(\g[-1]), V)$, where
\[
S_n \in \hom^{2-2n}_k (S^{>0}(\g[-1]), V)  \qquad
\text{for each $n \ge 0$}
\]
subject to 
\[
S_n (S^{>n+1}(\g[-1])) = 0 \qquad \text{for each $n \ge 0$}.
\]

\begin{theorem}
\label{second}
Given a $\BVi$-algebra $(V,\dhat_V)$, the associated quantum deformation
functor
\begin{equation*}
\QM_V: \LiAlg^{\op} \to \Set
\end{equation*}
is represented by the $\BVi$-algebra $V$ in the category of
$\BVi$-algebras or by the $\BVi$-space $\Spec V^*$ in the equivalent
category of $\BVi$-spaces. In other words, there is a natural
isomorphism
\[
\begin{split}
\QM_V (\g) & \xrightarrow{\sim} \Mor_{\BViSp} (\Spec S(\g[-1])^*, \Spec
V^*)\\ & := \Hom_{\BViAlg} ( (S(\g[-1]), \dhat_g), (V, \dhat_V)).
\end{split}
\]
\end{theorem}

\begin{proof}
The proof is almost a tautology: one just needs to observe that the
equation $\dhat_V \circ \exp (S/\hbar) = \exp (S/\hbar) \circ \dhat_\g$
defining $S$ as a $\BVi$-morphism is equivalent to the QME $\Dh
e^{S/\hbar} = 0$ for the $\BVi$-algebra $\hom_k (S(\g[-1]), V)$.
\end{proof}

\subsection{Quantum deformation functor associated to a bi-dg-Lie algebra}

Not\-with\-stand\-ing the apparent consistency of the quantum
deformation setup in the previous section, actual examples of quantum
deformations, see Section \ref{examples}, require certain modification
of the quantum deformation functor. Suppose $S(\g[-1])$ is the dg-BV
algebra arising from a bi-dg-Lie algebra $\g$, as in Example
\ref{bi-dg}. In this case, $\g$ is a bi-dg-Lie subalgebra of
$S(\g[-1])[1]$, the QME \eqref{QME} in $S(\g[-1])$ restricts to an
equation in $\g$, and the following subfunctor of the quantum
deformation functor becomes important:
\begin{equation}
\label{QMF}
\QM_\g (R) := \{ S \in \g[[\hbar]]^1 \, \widehat{\otimes}\; \m \; | \;
dS + \hbar \Delta S + \frac{1}{2} [S,S] = 0 \}.
\end{equation}
Note that $\g[[\hbar]]^1 = S^1 (\g[-1])[[\hbar]]^{2}$ and we have a
natural inclusion of functors $\QM_\g (R) \subseteq \QM_{S(\g[-1])}
(R)$. On the other hand, we have a natural identification
\[
\QM_\g (R) = \MC_{\g[[\hbar]]} (R),
\]
where $\g[[\hbar]]$ is considered as a dg-Lie algebra over $k$ with a
differential $\dhat = d + \hbar \Delta$. Thus, Theorems \ref{Q} and
\ref{Chuang-Lazarev} are applicable and we can state the following
easy corollary.

\begin{corollary}
Given a bi-dg-Lie algebra $\g$, the quantum deformation functor
$\QM_g: \FLAff^{\op} \to \Set$ is representable by the pointed formal
dg-manifold $\g[1][[\hbar]]$, and so is the extension of this functor
to the category of pointed formal dg-manifolds over $k[[\hbar]]$. In
other words, we have natural isomorphisms
\begin{gather*}
\QM_\g (R) \xrightarrow{\sim} \Mor_{\PFDGMan} (\Spec R,
\g[[\hbar]][1])\\ \QM_\g (\g') \xrightarrow{\sim}
\Mor_{\PFDGMan/k[[\hbar]]} (\g'[[\hbar]][1], \g[[\hbar]][1]),
\end{gather*}
for a complete local algebra $R$ and a bi-dg-Lie algebra $\g'$ or a
more general $\Li$-algebra $\g'[[\hbar]]$ over $k[[\hbar]]$.
\end{corollary}

\subsection{Examples of quantum deformations}
\label{examples}

Now we can discuss examples of quantum deformations, described by
solutions of QME in appropriate bi-dg-Lie and $\BVi$-algebras.

\begin{example}
\label{twisted-End}
Let $\OO$ be a modular operad, $V$ be a dg-vector space with
finite-dimensional graded components and an \emph{inner product of
  degree $-1$}, \emph{i.e}., a nondegenerate linear map $S^2(V) \to
k[-1]$, and $\End_V$, $\End_V ((g,n)) := V^{\otimes n}$, be the
endomorphism twisted modular operad of $V$. Consider the tensor
product $\OO \otimes \End_V$, which is a twisted modular operad with
components $(\OO \otimes \End_V) ((g,n)) := \OO((g,n)) \otimes \End_V
((g,n))$. Barannikov in \cite{barannikov}, cf.\ also \cite{kwz},
constructs, in fact, a bi-dg-Lie algebra structure $(\g, [-,-], d,
\Delta)$, see Example \ref{bi-dg}, on a shifted direct sum $\g :=
\bigoplus_{g,n} (\OO \otimes \End_V) ((g,n))_{S_n} [1] =
\bigoplus_{g,n} \OO((g,n)) \otimes_{S_n} V^{\otimes n}[1]$ of
$S_n$-coinvariants of the components of $\OO \otimes \End_V$. As we
know from Example \ref{bi-dg}, the bi-dg-Lie algebra $\g$ gives rise
to a dg-BV-algebra $S(\g[-1])$.
% the symmetric algebra $S(U)$ Let

According to Barannikov \cite{barannikov}, see also \cite{kwz},
solutions $S \in \g[[\hbar]]^1$ of the Quantum Master Equation
\begin{equation}
\label{QME-no-m}
d S + \hbar \Delta S + \frac{1}{2} [S,S] = 0
\end{equation}
are in bijection with $\F(\OO)$-algebra structures on $V$, where
$\F(\OO)$ is the Feynman transform \cite{getzler-kapranov} of the
modular operad $\OO$. Thus, we may think of the quantum deformation
functor \eqref{QMF} describing deformations of the trivial
$\F(\OO)$-algebra structure on $V$ corresponding to the trivial
solution $S = 0$ of the QME. Deformations of the $\F(\OO)$-algebra
corresponding to a nontrivial solution $S_0 \in \g[[\hbar]]^1$ of
\eqref{QME-no-m} may be described by solutions of the QME
\[
{\dhat}' S + \frac{1}{2} [S,S] = 0, \qquad S \in \g[[\hbar]]^1,
\]
with ${\dhat}' = d + \hbar \Delta + [S_0, -]$. As in Example
\ref{bi-dg}, the bi-dg-Lie algebra $(\g[[\hbar]], d + [S_0,-], \Delta,
    [-,-])$ over $k[[\hbar]]$ gives rise to a $\BVi$-algebra
    $S(\g[-1])$ with the BV operator ${\dhat}' = d + \{S_0, -\} + \hbar
    \Delta$ being a formal power series in $\hbar$ in which all the
    terms but those by $\hbar^1$ are derivations.

The simplest example of a modular operad $\OO$ is the modular envelope
of the commutative operad: $\OO((g,n)) := k$ for all $(g,n)$ in the
stable range. The corresponding notion of an $\F(\OO)$-algebra was
studied by Markl \cite{markl:loop}, who called it a \emph{loop
  homotopy Lie algebra}. It is a modular analogue of the (properadic)
notion of an $\IBLi$-algebra, which we will look at later.

Another standard example of a modular operad is the homology operad
$\OO((g,n)) \linebreak[0] = H_\bullet (\overline{\M}_{g,n}; k)$ of the
Deligne-Mumford moduli spaces $\overline{\M}_{g,n}$ of stable
algebraic curves of genus $g$ with $n$ punctures with respect to
attaching. In this case, the notion of an $\F(\OO)$-algebra will be a
higher-genus, homotopy extension of the notion of a \emph{gravity
  algebra}, see \cite{getzler:moduli95}.
\end{example}

\begin{example}
\label{twisted-O}
This is a twisted version of the previous example. Let $\OO$ be a
twisted modular operad, $V$ be a dg-vector space with
finite-dimensional graded components and an inner product of degree 0,
and $\End_V$, $\End_V ((g,n)) := V^{\otimes n}$, be the endomorphism
modular operad of $V$. Consider the tensor product $\OO
\otimes \End_V$, which is a twisted modular operad with components
$(\OO \otimes \End_V) ((g,n)) := \OO((g,n)) \otimes \End_V
((g,n))$. Again, Barannikov in \cite{barannikov}, cf.\ \cite{kwz},
constructs, a bi-dg-Lie-algebra structure $(\g, [-,-], d, \Delta)$,
see Example \ref{bi-dg}, on a shifted direct sum $\g :=
\bigoplus_{g,n} (\OO \otimes \End_V) ((g,n))_{S_n} [1] =
\bigoplus_{g,n} \OO((g,n)) \otimes_{S_n} V^{\otimes n}[1]$ of
$S_n$-coinvariants of the components of $\OO \otimes \End_V$. The
bi-dg-Lie algebra $\g$ is part of a dg-BV-algebra $S(\g[-1])$ up to
shift.
% the symmetric algebra $S(U)$ Let

Again, as per \cite{barannikov,kwz}, solutions of the Quantum Master
Equation
\[
d S + \hbar \Delta S + \frac{1}{2} [S,S] = 0
\]
in $\g[[\hbar]]^1$ are in bijection with $\F(\OO)$-algebra structures
on $V$, where $\F(\OO)$ is the Feynman transform of the twisted modular
operad $\OO$. 

Deformations of the $\F(\OO)$-algebra corresponding to a nontrivial
solution $S_0 \in \g[[\hbar]]^1$ of the QME are obtained by redefining
${\dhat}' := \dhat + [S_0, -]$ and considering the QME
${\dhat}' S + \frac{1}{2} [S,S] = 0$, $S \in \g[[\hbar]]^1$. As
in the previous example, this leads to a $\BVi$-algebra of a
particular type, with all the BV operators $\Delta_n$, except
$\Delta_2$, being derivations.

The model example of a twisted modular operad is the homology operad
$\OO((g,n)) \linebreak[0] = H_\bullet (\M_{g,n}; k)$ of the moduli
spaces $\M_{g,n}$ of algebraic curves of genus $g$ with $n$ punctures
with respect to twist-gluing. In this case, the notion of an
$\F(\OO)$-algebra will be a higher-genus, homotopy extension of the
notion of a \emph{hypercommutative}, or \emph{WDVV-algebra}, see
\cite{getzler:moduli95} and \cite{me:WDVV}.

There are various versions of the moduli-space example of a twisted
modular operad. One, in a different language, appeared in the work of
B.~Zwiebach \cite{zwiebach} and A.~Sen and Zwiebach
\cite{sen-zwiebach}. Translated to the language of our paper, they
considered $S^1$-equivariant chains on the moduli space of Riemann
surfaces with holomorphic disks. This gives a twisted modular
dg-operad. Another version of this operad, which uses the real version
of the Deligne-Mumford compactification, originated in the paper
\cite{ksv} of T.~Kimura, Stasheff and myself. See also Costello
\cite{costello:tft}. A solution to the QME in these cases is regarded
as a universal topological quantum field theory, the 2d, chain-level,
closed-string, nonperturbative flavor to be more precise.
\end{example}

\section{Quantizing Deformation Theory I}
\label{IBL-qd}

Recall from Example \ref{IBL} that an involutive Lie bialgebra $\g$
gives rise to a dg BV-algebra structure on the symmetric algebra
$S(\g[-1])$. Not surprisingly, a general $\BVi$-algebra structure on
the symmetric algebra $S(\g[-1])$ of a suspended graded vector $\g$ is
equivalent to the structure of a \emph{homotopy involutive Lie
  bialgebra}, called an $\IBLi$-\emph{algebra}, on $\g$. This is
actually the definition thereof, see \cite{cieliebak-fukaya-latschev}!
Moreover, the notion is equivalent to that of an
$\Omega(\CoFrob)$-algebra as per \cite[Theorem
  4.10]{drummond-cole-terilla-tradler}. Here $\CoFrob$ is a certain
co-Frobenius coproperad and $\Omega$ is the cobar construction,
producing a dg-properad. The notion of an $\IBLi$-algebra generalizes
not only that of an involutive Lie bialgebra but also the notion of a
bi-dg-Lie algebra, which played an important role in the previous
section. Indeed, both the involutive Lie bialgebra and bi-dg-Lie
algebra structures on a graded vector space $\g$ induce
dg-BV-structures on the symmetric algebra $S(\g[-1])$, the
Chevalley-Eilenberg complex of $\g$, see Examples \ref{IBL} and
\ref{bi-dg}.

In \emph{Quantizing Deformation Theory} \cite{terilla}, Terilla
conjectured the existence of quantized deformation theory, in which
commutative $k$-algebras $R$ would be replaced with (commutative)
Frobenius algebras and the Maurer-Cartan equation in an $\Li$-algebra
would be replaced with a master equation in an $\IBLi$-algebra
$\g$. The rationale is that the properad $\Frob$ describing Frobenius
algebras is a unit in the monoidal category of properads, just like
the operad $\Com$ describing commutative algebras is a unit in the
monoidal category of operads. Equivalently, if $V$ is an algebra over
a properad $\P$ and $F$ is a Frobenius algebra, then $V \otimes F$ is
again a $\P$-algebra.  Moreover, an $\IBLi$-algebra is equivalent to
an algebra over the dg-properad $\Omega(\CoFrob)$, whereas an
$\Li$-algebra is equivalent to an algebra over the operadic cobar
construction $\Omega(\CoCom)$ for the cocommutative co-operad
$\CoCom$.

This is an extremely striking analogy, but the current paper falls
short of proving Terilla's conjecture. However, I would like to
convince the reader that staying within the good old world of
deformations over commutative parameters still produces an interesting
quantization of deformation theory.

The matter is that there is a subtle difference between extending the
structure of an algebra over a properad $\P$ on a dg-vector space $V$
to a tensor product $V \otimes_k R$ and making a base change from $k$
to $R$. The structure of a commutative $k$-algebra on $R$ is not
enough to define the structure of a $\P$-algebra on $V \otimes R$ over
$k$. As mentioned above, endowing $R$ with a Frobenius-algebra
structure will suffice. On the other hand, when making a base change,
we are rather interested in a $\P \otimes R$-algebra structure on
$V \otimes R$ over $R$, which is always there as long as $R$ is a
commutative $k$-algebra. At the level of operations, not every
properadic operation with values, say, in the tensor square
$V \otimes V$ extends to an operation with values in
$(V \otimes R) \otimes (V \otimes R)$ but it does, if all we want is
an operation with values in $(V \otimes R) \otimes_R (V \otimes R)$.

What this means in the case of an $\IBLi$-algebra $\g$ and a complete
local algebra $R$ is that $\g \otimes R$ and $\g \widehat{\otimes} R$
are also $\IBLi$-algebras over $R$. Likewise, if $\g'$ is an
$\Li$-algebra, then $\hom_k (S(\g'[-1]), \g)$ is an $\IBLi$-algebra
over the dg-commutative algebra $S(\g'[-1])^*$. Accordingly,
$S(\g[-1]) \otimes R$, $S(\g[-1]) \widehat{\otimes} R$ and
$\hom_k (S(\g'[-1]), S(\g[-1]))$ are $\BVi$-algebras over
(dg-)commutative algebras $R$ and $S(\g'[-1])^*$, respectively. Thus,
the functors $\QM_V (R)$, see \eqref{QM}, and $\QM_V(\g')$, see
\eqref{QMV}, for $V = S(\g[-1])$ are well-defined and Theorems
\ref{first} and \ref{second} show that these functors are
representable.

\bibliographystyle{amsalpha}
\bibliography{qdt}

\end{document}